\documentclass[preprint]{elsarticle}

\usepackage{amssymb}    
\usepackage{amsmath}    
\usepackage{amsthm}     
\usepackage{hyperref}
\usepackage{xypic}

\newtheorem{theorem}{Theorem}[section]
\newtheorem{lemma}[theorem]{Lemma}
\newtheorem{proposition}[theorem]{Proposition}
\newtheorem{corollary}[theorem]{Corollary}

\newtheorem{predefinition}[theorem]{Definition}

\newtheorem{preremark}[theorem]{Remark}
\newenvironment{remark}{\begin{preremark}\rm}{\end{preremark}}
\newtheorem{prenotation}[theorem]{Notation}
\newenvironment{notation}{\begin{prenotation}\rm}{\end{prenotation}}
\newtheorem{preexample}[theorem]{Example}
\newenvironment{example}{\begin{preexample}\rm}{\end{preexample}}
\newtheorem{prequestion}[theorem]{Question}

\newcommand \CO {{\cal O}}
\newcommand \ZZ {{\mathbb Z}}
\newcommand \NN {{\mathbb N}}
\newcommand  \FF {{\mathbb F}}
\newcommand \Aut {\mathop{\rm Aut}}

\begin{document}

\begin{frontmatter}
\title{Wild cyclic-by-tame extensions}

\author[Penn]{Andrew Obus\corref{cor} \fnref{NDSEG}}
\ead{obusa@math.upenn.edu}
\author[Colostate]{Rachel Pries\fnref{nsf}}
\ead{pries@math.colostate.edu}
\fntext[NDSEG]{Supported by an NDSEG Graduate Research Fellowship}
\fntext[nsf]{Partially supported by NSF grant DMS-07-01303}
\cortext[cor]{Corresponding author}

\address[Penn]{Department of Mathematics, University of Pennsylvania, 209 S. 33rd Street, Philadelphia, PA, 19104}
\address[Colostate]{Department of Mathematics, Colorado State University, 101 Weber Building, Fort Collins, CO, 80523}

\begin{abstract}
Suppose $G$ is a semi-direct product of the form $\ZZ/p^n \rtimes \ZZ/m$ where $p$ is prime and
$m$ is relatively prime to $p$.  
Suppose $K$ is a complete local field of characteristic $p > 0$ with algebraically closed residue field.
The main result states necessary and sufficient conditions on the ramification filtrations that occur for
wildly ramified $G$-Galois extensions of $K$.
In addition, we prove that there exists a parameter space for $G$-Galois extensions of $K$
with given ramification filtration, and we calculate its dimension in terms of the ramification filtration.
We provide explicit equations for wild cyclic extensions of $K$ of degree $p^3$. 
\end{abstract}
\end{frontmatter}

\section{Introduction}

This paper is about wildly ramified Galois extensions of a complete local field $k((t))$ where
$k$ is an algebraically closed field of characteristic $p > 0$.
We prove that the lower jumps of the ramification filtration of a Galois extension of $k((t))$ 
with group $\ZZ/p^n \rtimes \ZZ/m$ are all congruent modulo $m$, Proposition \ref{Pcongruence}.  
We also prove that one can dominate a given Galois extension having
group $\ZZ/p^{n-1} \rtimes \ZZ/m$ by a Galois extension having group $\ZZ/p^n \rtimes \ZZ/m$,
with control over the last jump in the ramification filtration, 
Proposition \ref{Pdominate}.
Together with well-known results about ramification filtrations of Galois extensions with group $\ZZ/p^n$ \cite{Schmid}, 
this yields (see Theorem \ref{Tnecsuff}):

\begin{theorem} \label{Tmaintheorem}
Let $G$ be a semi-direct product of the form $\ZZ/p^n \rtimes \ZZ/m$ where $p \nmid m$.
Let $\sigma \in G$ have order $p^n$ and let $m'=|{\rm Cent}_G(\sigma)|/p^n$.
A sequence $u_1 \leq \cdots \leq u_n$ of rational numbers
occurs as the set of positive breaks in the upper numbering of the ramification filtration
of a $G$-Galois extension of $k((t))$ if and only if:
\begin{description}
\item{(a)} $u_i \in \frac{1}{m} \NN$ for $1 \leq i \leq n$;
\item{(b)} $\gcd(m, mu_1)=m'$;
\item{(c)} $p \nmid mu_1$ and, for $1 < i \leq n$, either $u_i=pu_{i-1}$ or both $u_i > pu_{i-1}$ and $p \nmid mu_i$;
\item{(d)} and $mu_i \equiv mu_1 \bmod m$ for $1 \leq i \leq n$.
\end{description}
\end{theorem}

In the first author's doctoral thesis, Theorem \ref{Tmaintheorem} yields restrictions on the stable reduction of 
certain branched covers of the projective line.

Our other main result, Theorem \ref{Tparameter}, states that, given a group $G$ and
a ramification filtration $\eta$ satisfying conditions (a)-(d) as in Theorem \ref{Tmaintheorem}, 
there exists a parameter space ${\mathcal M}_\eta$ whose $k$-points
are in natural bijection with isomorphism classes of $G$-Galois extensions of $k((t))$ having ramification filtration $\eta$.
We calculate the dimension of ${\mathcal M}_\eta$ in terms of the upper jumps of $\eta$.

Here is the paper's outline: in Section \ref{Sbackground} we introduce the framework of study, 
including ramification filtrations and field theory;
Section \ref{Sallcyclic} contains several structural descriptions of cyclic $p$-group extensions;
in Section \ref{SCycByTame}, we prove results about tame actions on cyclic extensions;
and the main results on ramification filtrations and parameter spaces for $G$-Galois extensions
appear in Section \ref{Smain}.

Our original motivation for this topic was to find explicit equations for $\ZZ/p^3$-Galois extensions of $k((t))$, 
see Section \ref{Sp3}.
Such equations are useful and are difficult to find in the literature.
For example, in \cite[II, Lemma 5.1]{GM},
the authors use equations for $\ZZ/p^2$-Galois extensions
in order to prove a case of Oort's Conjecture,
namely, that every $\ZZ/p^2$-Galois extension of $k((t))$ lifts to characteristic $0$ \cite[Thm.\ 2]{GM}.

Similar results for elementary abelian $p$-group extensions are in \cite{Deo}.

We thank D. Harbater and an anonymous reader for help with Proposition \ref{Pcongruence}, 
and J. Achter, S. Corry, G. Elder, M. Matignon, and the referee for useful comments.

\section{Framework of study} \label{Sbackground}

This section contains background on extensions of complete local fields and ramification filtrations
and introduces the situation studied in this paper, in which 
the Galois group is a semi-direct product of the form $\ZZ/p^n \rtimes \ZZ/m$.

\subsection{Extensions of complete local fields} \label{Slocal}

Let $k$ be an algebraically closed field of characteristic $p > 0$.
We fix a compatible system of roots of unity of $k$.
In particular, this fixes an isomorphism $\ZZ/p \simeq \FF_p$
and fixes a primitive $m$th root of unity $\zeta$ in $k$.
Let $R$ be an equal characteristic complete discrete valuation ring with residue field $k$ and fraction field $K$.
Then $R \simeq k[[t]]$ and $K \simeq k((t))$ for some uniformizing parameter $t$.

Suppose $L/K$ is a separable Galois field extension with group $G$.
Let $S$ be the integral closure of $R$ in $L$.
Then $S/R$ is a Galois extension of rings with group $G$ which is totally ramified over the prime ideal $(t)$.

This type of field extension arises in the following context.
Suppose $\phi:Y \to X$ is a Galois cover of smooth $k$-curves.  
Suppose $y \in Y$ is a ramified point with inertia group $G$.  
Consider the complete local rings $S=\hat{\CO}_{Y,y}$ and $R=\hat{\CO}_{X,\phi(y)}$.
Then $S/R$ is a Galois extension of rings with group $G$ which is totally ramified over the unique valuation of $R$
as described in the preceding paragraph.

For a Galois extension $L/K$ as above, the group $G$ is a semi-direct product of the 
form $P \rtimes \ZZ/m$ where $P$ is a $p$-group and $p \nmid m$ \cite[IV, Cor.\ 4]{Se:lf}.
Throughout the paper, we assume that the subgroup $P$ is cyclic.

\subsection{Subgroups of a semi-direct product} \label{Sgroup}

Suppose $G$ is a semi-direct product of the form $P \rtimes \ZZ/m$ where $P \simeq \ZZ/p^n$
and $p \nmid m$.  Let $\sigma$ be a chosen generator of $P$.
Let $c$ be a chosen element of order $m$ in $G$ and let $M=\langle c \rangle$.
Let $m'=|{\rm Cent}_G(\sigma)|/p^n$.  
In other words, $m'=\#\{g \in M \ | \ g\sigma g^{-1} = \sigma\}$.

For $0 \leq i \leq n$, the element $\sigma_i:=\sigma^{p^i}$ has order $p^{n-i}$
and $H_i:=\langle \sigma_i \rangle$ is the unique subgroup of order $p^{n-i}$ in $G$.
Then $\{{\rm id}\}=H_n \subset H_{n-1} \subset \cdots \subset H_0=P$.

The semi-direct product is determined by the conjugation action of $M$ on $P$.
Since $c \sigma c^{-1}$ also generates $P$, then $c \sigma c^{-1}=\sigma^{\alpha'}$ for some
integer $\alpha'$ such that $1 \leq \alpha' < p^n$ and $p \nmid \alpha'$.
The action of $c$ stabilizes $H_i$.  Let $J_i:=(H_{i-1}/H_i) \rtimes M$.

\begin{lemma} \label{Lalpha}
\begin{description}
\item{(i)} The value of $\alpha'$ does not depend on the choice of generator of $P$;
\item{(ii)} The value of $\alpha'$ depends on the choice of generator of $M$ as follows;
if $c_0=c^{\beta}$ for some integer $\beta$, 
then $\alpha_0' \equiv (\alpha')^\beta \bmod p^n$.
\end{description}
\end{lemma} 

\begin{proof}
\begin{description}
\item{(i)} If $\tau=\sigma^\gamma$, 
then $c\tau c^{-1}=(c\sigma c^{-1})^\gamma=(\sigma^{\alpha'})^\gamma=\tau^{\alpha'}$.
\item{(ii)} By induction, $c^i \sigma c^{-i}=\sigma^{(\alpha')^i}$.  Thus $c_0 \sigma c_0^{-1}=\sigma^{\alpha'_0}$.
\end{description}
\end{proof}

\begin{lemma} \label{Lgroup}
The groups $J_i$ are canonically isomorphic for $1 \leq i \leq n$.
\end{lemma}

\begin{proof}
The groups $J_i$ are semi-direct products of the form $\ZZ/p \rtimes \ZZ/m$.  
Thus it suffices to show that the action of $c$ on the equivalence class of $\sigma_{i-1}$ modulo 
$\langle \sigma_i \rangle$ is the same for $1 \leq i \leq n$. 
Note that $c \sigma^p c^{-1}=(\sigma^p)^{\alpha'}$.
Thus $c \sigma_i c^{-1} = \sigma_i^{\alpha'}$.    
\end{proof}

The residue of $\alpha'$ modulo $p$ can be identified with an element $\alpha \in (\ZZ/p)^*$
and thus with an element $\alpha \in \FF^*_p$.
Also $m/m'$ is the order of $\alpha$ in $\FF_p^*$.

\subsection{Towers of fields} \label{Sfields}

Suppose $L/K$ is a separable Galois extension 
whose group $G$ is of the form $\ZZ/p^n \rtimes \ZZ/m$ with $p \nmid m$.
We fix an identification of $\Aut(L/K)$ with $G$ and indicate this by writing that $L/K$ is a
{\it $G$-Galois extension}.

Consider the fixed fields $L_i=L^{H_i}$ and $K_i=L^{H_i \rtimes M}$ for $0 \leq i \leq n$.
So, $L_n=L$ and $K_0=K$.
Let $v_i$ be the natural valuation on $L_i$.
Let $\Theta_i$ be the integral closure of $R$ in $L_i$.  
Then $L/L_i$ is an $H_i$-Galois extension
and $L_i/L_0$ is a $P/H_i$-Galois extension.
Also $L_i/K_{i-1}$ is a $J_i$-Galois extension.
This yields a tower of fields:

\[
\xymatrix{
L_0  \ar@{^{(}->}[r]^{\ZZ/p} & L_1  \ar@{^{(}->}[r]^{\ZZ/p} & \cdots  \ar@{^{(}->}[r]^{\ZZ/p} & L_{n-1}  \ar@{^{(}->}[r]^{\ZZ/p} &  L
\\
K_0  \ar@{^{(}->}[r]  \ar@{^{(}->}[u]^{\ZZ/m} & K_1  \ar@{^{(}->}[r]  \ar@{^{(}->}[u] & 
\cdots \ar@{^{(}->}[r] & K_{n-1}  \ar@{^{(}->}[r]  \ar@{^{(}->}[u]  & K_n  \ar@{^{(}->}[u]
}
\]

By Kummer theory, there exists $x \in L_0$ such that $L_0 \simeq K[x]/(x^m - 1/t)$. 
After choosing $c \in G$ such that $c(x)=\zeta x$, one can determine the values of 
$\alpha'$ and $\alpha$ for the extension $L/K$.

\subsection{Ramification filtrations} \label{Sram}

Here is a brief review of the theory of ramification filtrations from \cite[IV]{Se:lf}.
Consider the natural valuation $v=v_n$ on $L$
and a uniformizing parameter $\pi \in L$.
For $r \in \NN$, let $I_{r}$ be the $r$th ramification group in the lower numbering for the extension $L/K$.
In other words, $I_{r}$ is the normal subgroup of all $g \in G$
such that $v(g(\pi)-\pi) \geq r+1$.

The ramification filtration is important because it determines the degree $\delta$ of the different of $S/R$.
Namely, by \cite[IV, Prop.\ 4]{Se:lf}, $\delta= \sum_{r \geq 0} (|I_r|-1)$.
If $\phi:Y \to X$ is a cover of smooth projective connected $k$-curves,
the genus of $Y$ can by found using the Riemann-Hurwitz formula
\cite[IV, Cor. 2.4]{Har}
and this formula relies on the degree of the different at each ramification point of $\phi$.

Let $g \in G$ with $g \not = 1$.
The {\it lower jump} for $g$ is the non-negative integer $j$ so that $v(g(\pi)-\pi)= j+1$.
Then $g \in I_{j}$ and $g \not \in I_{j+1}$.
By \cite[IV, Prop.\ 11]{Se:lf}, $p \nmid j$ for any positive lower jump $j$.
If $|P|=p^n$, then there are $n$ positive indices $j_1 \leq \cdots \leq j_n$ at which there is a break in
the ramification filtration in the lower numbering, which are called the {\it lower jumps} of $L/K$.

There is also a ramification filtration $I^\ell$ in the upper numbering.  
The {\it upper jumps} of $L/K$ are the positive breaks $u_1 \leq \cdots \leq u_n$ 
in the ramification filtration in the upper numbering.
The lower numbering is stable for subextensions \cite[IV, Prop.\ 2]{Se:lf} and the upper numbering is stable
for quotients \cite[IV, Prop.\ 14]{Se:lf}.
Using Herbrand's formula \cite[IV, \S3]{Se:lf}, one can translate between the two
ramification filtrations: letting $j_0=u_0=0$, then $u_{i}-u_{i-1}=(j_{i}-j_{i-1})/p^{i-1}m$ for $1 \leq i \leq n$.

\section{Wild cyclic extensions} \label{Sallcyclic}

In this section, we describe the equations and ramification filtration of the $\ZZ/p^n$-Galois subextension $L/L_0$.  
The material in this section is mostly known, but it is all necessary for later results in the paper.

\subsection{Cyclic towers of Artin-Schreier extensions} \label{Scyclic}

\begin{lemma} \label{Ljump}    
The $i$th lower jump $j_i$ of $L/K$ equals the lower jump of $L_i/L_{i-1}$.
\end{lemma}

\begin{proof}  
The $i$th lower jump $j_i$ of $L/K$ is the lower jump of
the automorphism $\sigma_{i-1}$.
This is the same as the lower jump of $\sigma_{i-1}$ for the extension $L/L_{i-1}$ by \cite[IV, Prop.\ 2]{Se:lf}.
Since this is the smallest lower jump for the extension $L/L_{i-1}$, it also equals the upper jump of
$\sigma_{i-1}$ for $L/L_{i-1}$.  By \cite[IV, Prop.\ 14]{Se:lf}, 
this is then the same as the upper jump, and thus the lower jump, of $L_i/L_{i-1}$.
\end{proof}

\subsection{Witt Vectors and $p$-power cyclic extensions}\label{SWitt}

We recall some Witt vector theory.
Let $\wp$ be the operation ${\rm Fr} - {\rm Id}$ on Witt vectors, where ${\rm Fr}$ denotes Frobenius. 
An element $a$ of a field $F$ of characteristic $p$ 
is a $\wp${\it th power} in $F$ if the polynomial $z^p - z - a$ has a root in $F$.

By \cite[p.\ 331, Ex.\ 50]{La:Alg}, every Galois extension of $L_0 \cong k((x^{-1}))$
with group $\ZZ/p^n$ has Witt vector equations 
\begin{equation}\label{EWitt}
(y_1^p, \ldots, y_n^p) = (y_1, \ldots, y_n) +' (x_1, \ldots, x_n).
\end{equation}
where $x_i \in L_0$ for $1 \leq i \leq n$ such that $x_1$ is not a $\wp$th power in $L_0$ 
and where $+'$ denotes addition of Witt vectors:  
Moreover, there is a generator $\tau$ of $\ZZ/p^n$ such that the action of 
$\tau$ on Witt vectors is
\begin{equation}\label{Eaction}
\tau(y_1, \ldots, y_n)=(y_1, \ldots, y_n) +' (1, 0, \ldots, 0). 
\end{equation}

Modifying $(x_1, \ldots, x_n)$ by an element $w \in W^n(L_0)$, where $W^n$ is the $n$th truncation of the 
Witt vectors, changes the isomorphism class of the extension precisely when $w \not \in \wp(W^n(L_0))$.  
Thus, since $k$ is algebraically closed, one can choose $(x_1, \ldots, x_n)$ to be in {\it standard form}, 
i.e., $x_i \in k[x]$ and either $x_i =0$ or $x_i$ has no exponent divisible by $p$.  

To make (\ref{EWitt}) more explicit, for $0 \leq i \leq n-1$, let 
$W_i=\sum_{d=0}^i p^dX_{d+1}^{p^{i-d}}$ be the $i$th Witt polynomial, \cite[II, \S6]{Se:lf}.  
Define $S_i \in \ZZ[X_1, \ldots, X_{i+1}, Y_1, \ldots, Y_{i+1}]$ to be the unique formal polynomial such that
\begin{eqnarray*} 
\lefteqn{W_i(X_1, \ldots, X_{i+1}) + W_i(Y_1, \ldots, Y_{i+1}) =} \\
& & W_i(S_0(X_1, Y_1), S_1(X_1, X_2, Y_1, Y_2), \ldots, S_i(X_1, \ldots, X_{i+1}, Y_1, \ldots Y_{i+1})).
\end{eqnarray*}
The indexing of these variables is shifted by one from that of \cite[II, \S6]{Se:lf} in order
to be more consistent with notation in this paper.
By \cite[II, Thm.\ 6]{Se:lf}, the $S_i$ are well defined and have integer coefficients.  

\begin{lemma}\label{LWitt}
In $\ZZ[X_1, \ldots, X_i, Y_1, \ldots, Y_i]$, 
$$S_{i-1}(X_1, \ldots, X_{i}, Y_1, \ldots, Y_{i}) = X_{i} + Y_{i} + 
\sum_{d=1}^{i-1} \frac{1}{p^{i-d}} (X_{d}^{p^{i-d}} + Y_{d}^{p^{i-d}} - S_{d-1}^{p^{i-d}})$$
and the degree of every monomial of $S_{i-1}$ is congruent to one modulo $p-1$.
\end{lemma}

\begin{proof} 
The equation follows from $\sum_{d=0}^{i-1} p^d S_d^{p^{i-1-d}} = \sum_{d=0}^{i-1} p^d(X_{d+1}^{p^{i-1-d}} + Y_{d+1}^{p^{i-1-d}})$  
(see \cite[Footnote 4]{Schmid}) and the statement about degrees from induction.
\end{proof}

For $1 \leq i \leq n$, let $\bar S_{i-1} \in \FF_p[X_1, \ldots, X_i, Y_1, \ldots, Y_i]$ be the reduction of $S_{i-1}$ modulo $p$
and let $f_i(Y_1, \ldots, Y_{i-1}, X_1, \ldots X_i)=\bar S_{i-1}-Y_i$.
Then $f_i=X_i+g_i$ where $g_i \in \FF_p[X_1, \ldots, X_{i-1}, Y_1, \ldots, Y_{i-1}]$ is a polynomial whose terms
each have degree congruent to one modulo $p-1$.
The meaning of (\ref{EWitt}) is that a Galois extension with group $\ZZ/p^n$ has equations 
$y_i^p-y_i=f_i(y_1, \ldots, y_{i-1}, x_1, \ldots x_i)$.

\begin{lemma} \label{Lgalois}
Let $L/L_0$ be a $\ZZ/p^n$-Galois extension and $\sigma$ a generator of $\ZZ/p^n$. 
There exist $x_i \in L_0$ and $y_i \in L$ for $1 \leq i \leq n$ such that $L/L_0$ is isomorphic
to the $\langle \sigma \rangle$-Galois extension with Witt vector equations and Galois action
$$(y_1^p, \ldots, y_n^p) = (y_1, \ldots, y_n) +' (x_1, \ldots, x_n)$$
$$\sigma(y_1, \ldots, y_n) = (y_1, \ldots, y_n) +' (1, 0, \ldots, 0).$$
Furthermore, there is a unique choice for $(x_1, \ldots, x_n)$ in standard form.
\end{lemma}

\begin{proof}
There exist $x_i \in L_0$ and $y_i \in L$ and a generator $\tau$ of $\ZZ/p^n$ such that
$L/L_0$ has Witt vector equations (\ref{EWitt}) and Galois action (\ref{Eaction}).
Now $\sigma = \tau^b$ for some $b \in (\ZZ/p^n)^*$.  
Then $\sigma(y_1, \ldots, y_n) = (y_1, \ldots, y_n) +' b(1, 0, \ldots, 0)$.  
Since $b$ is invertible in $\ZZ/p^n \cong W^n(\ZZ/p) \subset W^n(L_0)$,   
one can replace $(y_1, \ldots, y_n)$ and  $(x_1, \ldots, x_n)$ with the Witt vectors $\frac{1}{b} (y_1, \ldots, y_n)$ 
and $\frac{1}{b} (x_1, \ldots, x_n)$.  
Since ${\rm Fr}$ is a ring homomorphism \cite[p.\ 331, Ex.\ 48]{La:Alg}, 
the extension $L/L_0$ still has Witt vector equations (\ref{EWitt})
and now $\sigma(y_1, \ldots, y_n) = (y_1, \ldots, y_n) +' (1, 0, \ldots, 0)$.

By a generalization of \cite[Lemma 2.1.5]{Pr:deg}, there is a unique choice of $(x_1, \ldots, x_n)$
in standard form compatible with the restriction on the Galois action.
\end{proof}

\subsection{Ramification filtrations for cyclic $p$-group extensions}

The ramification filtration of a $\ZZ/p^n$-Galois extension
is completely determined by either its lower or upper jumps, which in turn 
can be determined by the Witt vector equation.  
 
\begin{lemma} \label{Lconductor}
Let $L/L_0$ be a $\ZZ/p^n$-Galois extension with Witt vector $(x_1, \ldots, x_n)$ in standard form.
Let $u={\rm max}\{-p^{n-i}v_0(x_i)\}_{i=1}^n$.
Then $u$ is the last upper jump of $L/L_0$.
\end{lemma}

\begin{proof}
This follows from \cite[Thm.\ 1.1]{Ga:ls}; see also \cite[Prop. 4.2(1)]{Th:ASW}.
\end{proof}

We retrieve the following classical result.

\begin{lemma} \label{Lcyclicjumps}
A sequence of positive integers $w_1 \leq \cdots \leq w_n$
occurs as the set of upper jumps of a $\ZZ/p^n$-Galois extension of $L_0$ 
if and only if $p \nmid w_1$ and, for $1 < i \leq n$, either $w_i=pw_{i-1}$ or both $w_i > pw_{i-1}$ and $p \nmid w_i$.
\end{lemma}

\begin{proof}
The result, originally found in \cite{Schmid}, follows from Lemma \ref{Lconductor}; see also \cite[Lemma 19]{Pr:genus}.
\end{proof}

The following lemma will be used to compare the upper jumps of the $G$-Galois extension $L/K$ 
and the $\ZZ/p^n$-Galois extension $L/L_0$.

\begin{lemma} \label{Lcomparejumps}
Suppose $L/K$ has upper jumps $u_1 \leq \cdots \leq u_n$.  
Then $L/L_0$ has upper jumps $w_1 \leq \cdots \leq w_n$ where $w_i=mu_i$ for $1 \leq i \leq n$.
\end{lemma}

\begin{proof}
By  \cite[IV, Prop.\ 2]{Se:lf}, the lower jumps of $L/L_0$ equal the lower jumps $j_1 \leq \cdots \leq j_n$ of $L/K$.
Herbrand's formula \cite[IV, \S3]{Se:lf} implies that $u_{i}-u_{i-1}=(j_{i}-j_{i-1})/p^{i-1}m$ 
and that $w_{i}-w_{i-1}=(j_{i}-j_{i-1})/p^{i-1}$ for $1 \leq i \leq n$. 
\end{proof}

\section{Cyclic-by-tame extensions}\label{SCycByTame}

Suppose $L/K$ is a separable $G$-Galois field extension as in Sections \ref{Sgroup}-\ref{Scyclic}. 
In this section, we find necessary conditions on the ramification filtrations and equations arising from the 
$\ZZ/m$-Galois action on $L$.

\subsection{The case of Galois extensions with group $\ZZ/p \rtimes \ZZ/m$} \label{Sbase}

\begin{lemma} \label{Lgcd}
Consider the $J_1$-Galois extension $L_1/K$ with equations $x^m=1/t$ and $y_1^p-y_1=x_1$
and Galois action $c(x)=\zeta x$ and $\sigma(y_1)=y_1+1$.
\begin{description}
\item{(i)}
The lower jump $j$ of $L_1/L_0$ satisfies $m' = \gcd(m, j)$.
\item{(ii)}
Also $m | j(p-1)$.  In particular, $j \equiv jp^r \bmod m$ for any $r \in \NN$.
\item{(iii)} Also $c(y_1) = \alpha^{-1} y_1 = \zeta^{j}y_1$.  
\end{description}
\end{lemma}

\begin{proof}
\begin{description}
\item{(i)}
This follows from \cite[IV, Prop.\ 9]{Se:lf}, see also \cite[Lemma 1.4.1(iv)]{Pr:deg}.
\item{(ii)}
The conjugation action of $\ZZ/m$ on $\ZZ/p$ gives a homomorphism $\nu:\ZZ/m \to \Aut(\ZZ/p)$.
By definition, ${\rm Im}(\nu)$ has order $m/m'$ and ${\rm Ker}(\nu)=\langle c^{m/m'} \rangle$.
Thus $m| m'(p-1)$.  By part (i), $m' = \gcd(m, j)$, so $m | j(p-1)$.
\item{iii)} \cite[Lemma 1.4.1(ii)-(iii)]{Pr:deg}.
\end{description}
\end{proof}

\subsection{A congruence condition on the ramification filtration}\label{Scongruence}

\begin{proposition} \label{Pcongruence}
\begin{description}
\item{(i)}
The lower jumps in the ramification filtration of the $P$-Galois extension $L/L_0$ are all congruent modulo $m$.
\item{(ii)}
The upper jumps in the ramification filtration of the $P$-Galois extension $L/L_0$ are all congruent modulo $m$.
\end{description}  
\end{proposition}

\begin{proof}
\begin{description}

\item{(i)}
The $i$th lower jump of $L/L_0$ is $j_i$ by \cite[IV, Prop.\ 2]{Se:lf}.
Let $\pi$ be a uniformizer of $\Theta_n$ and let $u=c(\pi)/\pi \in \Theta_n^{*}$.  Then $u$ equals
$\theta_0(c) \in k^{*}$ in the notation of \cite[IV, Prop.\ 7]{Se:lf}.  
The order of $u$ is $m$ by \cite[IV, Prop.\ 7]{Se:lf}.  
By the proof of Lemma \ref{Lgroup}, $c \sigma_{i-1} c^{-1} = \sigma_{i-1}^{\alpha'}$ for $1 \leq i \leq n$.
Since $\sigma_{i-1}$ generates $H_{i-1}/H_{i}=I_{j_i}/I_{j_i+1} $, 
\cite[IV, Prop.\ 9]{Se:lf} shows that $\theta_{j_i}(\sigma_{i-1}^{\alpha'})=u^{j_i}\theta_{j_i}(\sigma_{i-1})$ for $1 \leq i \leq n$.
Thus $u^{j_i}=\alpha \in k^*$ for $1 \leq i \leq n$ and so $j_1 \equiv \cdots \equiv j_n \bmod m$.

\item{(ii)}
Let $w_1 \leq \cdots \leq w_n$ be the upper jumps of the $P$-Galois extension $L/L_0$.
Since $P$ is abelian, the Hasse-Arf Theorem implies that $w_i \in \NN$.
By Herbrand's formula, $w_i-w_{i-1} = (j_i-j_{i-1})/p^{i-1}$.
Thus $w_i-w_{i-1} \equiv 0 \bmod m$ by part (i).
\end{description}
\end{proof}

{\bf Class field theory approach:}
If $k$ is instead a finite field, here is a different proof of Proposition \ref{Pcongruence} which uses class field theory.

\begin{proof}[Second proof of Proposition \ref{Pcongruence}]
The $G$-Galois extension $L/K$ dominates the $\langle c \rangle$-Galois extension $L_0/K$ where
$L_0 \simeq k((x^{-1}))$, $x^m=1/t$, and $c(x)=\zeta x$.
Let $L/L_0$ be the $P$-Galois subextension, which has upper jumps 
$w_1 \leq \cdots \leq w_n$ where $w_i=mu_i$ by Lemma \ref{Lcomparejumps}.
Thus the upper ramification group $I^\ell$ of $L/L_0$ equals $H_{i}$ if $w_{i} < \ell \leq w_{i+1}$.   

Let $Q=(x^{-1})$ be the maximal ideal of $k[[x^{-1}]]$.
Consider the unit groups $U^d=1+Q^d$ of $k[[x^{-1}]]$ \cite[IV.2]{Se:lf}.
By \cite[IV, Prop.\ 6]{Se:lf}, $U^d/U^{d+1}$ is canonically isomorphic to $Q^d/Q^{d+1}$.
Now, $Q^d$ carries a natural $\langle c \rangle$-module structure where $c((x^{-1})^d)=\zeta_m^{-d} (x^{-1})^d$.
Thus $U^d/U^{d+1}$ carries a natural structure as a $\langle c \rangle$-module, and this structure 
depends on the congruence class of $d$ modulo $m$.

By \cite[XV.2, Cor.\ 3 \& pg.\ 229]{Se:lf}, 
there is a reciprocity isomorphism $\omega:L_0^*/NL^* \to P$ and thus 
there are isomorphisms $\omega_n: U^d/(U^{d+1}NU_L^{\psi(d)}) \to I^d/I^{d+1}$.
Here $N:L \to L_0$ is the norm map and $\psi$ is Herbrand's function.
In particular, taking $d=w_{i}$, then
$U^{w_i}/(U^{w_i+1}NU_L^{\psi(w_i)}) = H_{i-1}/H_i$.

Now $H_{i-1}/H_i$ has a $\langle c \rangle$-module structure and
this $\langle c \rangle$-module structure is independent of $i$ by Lemma \ref{Lgroup}.
After pulling back by $\omega$, 
this implies that the $\langle c \rangle$-module structure of $U^{w_i}/(U^{w_i+1}NU_L^{\psi(w_i)})$ 
and thus of $U^{w_i}$ is independent of $i$.
Thus $\zeta_m^{-w_i}$ is independent of $i$ and so $w_i \equiv w_1 \bmod m$.

The lower jumps are also congruent modulo $m$ by Herbrand's formula.
\end{proof} 

At this point, one can prove that the conditions in Theorem \ref{Tmaintheorem} are necessary;
we will postpone this until Section \ref{Smaintheorem}.

\subsection{Actions and isomorphisms}

This section contains two results that will be needed in Section \ref{Smain}.

\begin{proposition}\label{Paction}
Suppose $L_0 \simeq K[x]/(x^m-1/t)$ and $c(x)=\zeta x$.
Suppose $L/L_0$ is a $P$-Galois extension with Witt vector equations (\ref{EWitt}), 
Galois action (\ref{Eaction}), and first lower jump $j$ such that $\zeta^{j}=\alpha^{-1}$.
Then $L/K$ is a $G$-Galois extension if and only if 
$c(x_i) = \zeta^{j}x_i$ and $c(y_i) = \zeta^{j}y_i$ for $1 \leq i \leq n$.
\end{proposition}

\begin{proof}
Suppose $L/K$ is a $G$-Galois extension. 
Then $L_1/K$ is a $J_1$-Galois extension.
By Lemma \ref{Lgcd}(iii), $c(y_1)/y_1=\alpha^{-1}=\zeta^{j}$.
Since $y_1^p-y_1=x_1$, this implies that $c(x_1) = \zeta^{j}x_1$.
As an inductive hypothesis, suppose that $c(x_i) = \zeta^{j}x_i$ and $c(y_i) = \zeta^{j}y_i$ for $1 \leq i \leq n-1$.

Now $L_n/K_{n-1}$ is a $J_n$-Galois extension of local fields
and $J_n$ and $J_1$ are canonically isomorphic by Lemma \ref{Lgroup}.
In other words, the value of $\alpha$ for $\Aut(L_n/K_{n-1})$ is the same as for $\Aut(L_1/K)$.    
By Kummer theory, there exists a uniformizer $\pi_{n-1}$ of $L_{n-1}$ such that
$c$ acts on $\pi_{n-1}$ via multiplication by some $\gamma \in \mu_m$.
Then $L_n/K_{n-1}$ satisfies the hypotheses of Lemma \ref{Lgcd},
with $1/\pi_{n-1}$, $y_n$, $j_n$, and $\gamma^{-1}$ replacing $x$, $y_1$, $j$, and $\zeta$ respectively. 
Applying Lemma \ref{Lgcd}(iii) to $L_n/K_{n-1}$ implies that $c(y_n)/y_n=\gamma^{-j_n} =\alpha^{-1}=\zeta^j$.

The equation for $L_n/L_{n-1}$ is $y_n^p-y_n=x_n+g_n$ where the terms of the polynomial
$g_n \in \FF_p[x_1, \ldots, x_{n-1}, y_1, \ldots, y_{n-1}]$
each have degree congruent to one modulo $p-1$.
By the inductive hypothesis and Lemma \ref{Lgcd}(ii), $c$ scales $g_n$ by $\zeta^j$.
Thus $c$ scales both $y_n^p-y_n-x_n$ and $y_n$ by $\zeta^j$, which implies $c(x_n)=\zeta^j x_n$.

Conversely, suppose $c(x_i) = \zeta^{j}x_i$ and $c(y_i) = \zeta^{j}y_i$ for $1 \leq i \leq n$.
The proof that $L/K$ is $G$-Galois 
proceeds by induction on $n$; the case $n=1$ can be computed explicitly,  
see e.g.\ \cite[Lemma 1.4.1]{Pr:deg}.
As an inductive hypothesis, suppose that $L_{n-1}/K$ is a $G/H_{n-1}$-Galois extension.
To finish, it suffices to show that the
action of $c$ extends to an automorphism of $L_n$, i.e., 
that $c$ stabilizes the equation $y_n^p-y_n = f_n$ for $L_n/L_{n-1}$.
By Lemmas \ref{LWitt} and \ref{Lgcd}(ii), the action of $c$ scales every term of this equation by $\zeta^{j}$.
\end{proof}

\begin{lemma} \label{Lisom}
Suppose $L/K$ is a $G$-Galois extension as in Section \ref{Sfields}.
\begin{description}
\item{(i)}
There is a Witt vector $(x_1, \ldots, x_n)$ in standard form for the subextension $L/L_0$ 
and it is uniquely determined up to multiplication by $\mu_{m/m'}$.
\item{(ii)} 
There are $\varphi(m)/\varphi(m/m')$ different non-isomorphic 
$G$-Galois structures on the field extension $L/K$ 
such that the action of $\sigma$ on $L$ is as in (\ref{Eaction}). 
\end{description}
\end{lemma}

\begin{proof}
For part (i), 
by Lemma \ref{Lgalois}, for fixed $x$, there is a uniquely determined 
Witt vector $(x_1, \ldots, x_n)$ in standard form for the subextension $L/L_0$.  
Now $x$ is determined up to multiplication by $\zeta^d$, for $d \in \ZZ$. 
By Proposition \ref{Paction}, every monomial in $x_i$ has degree congruent to $j \bmod m$. 
Replacing $x$ with $\zeta^d x$ scales $x_i$ by $\zeta^{dj}$.
The values of $\zeta^{dj}$ range over $\mu_{m/m'}$ by Lemma \ref{Lgcd}(i).  

For part (ii),
a $G$-Galois structure on $L/K$ satisfying the requirement for $\sigma$ is determined by an 
isomorphism $\iota: G \to \Aut(L/K)$ such that 
$\iota(\sigma)(y_1, \ldots, y_n) = (y_1, \ldots, y_n) +' (1, 0, \ldots, 0)$.  
If $h \in \Aut(L/K)$, then the map $h: L \to L$ yields an isomorphism of
$G$-Galois extensions $L/K \to L/K$, the first with structure morphism $\iota$ and the second with structure morphism $h \iota h^{-1}$.  
Thus, modifying $\iota$ by an inner automorphism yields an isomorphic $G$-Galois structure on $L/K$.
So the number of isomorphism classes of $G$-Galois structures with this requirement on $\sigma$
is given by the number of elements of $\Aut(G)$ fixing $\sigma$, divided by the number of ${\rm Inn}(G)$ 
fixing $\sigma$.

An automorphism $\gamma$ of $G$ which fixes $\sigma$ is determined by $\gamma(c)$.  
Also $\gamma(c)$ must have order $m$ and have the same conjugation action as $c$ on $\sigma$, 
as determined by Lemma \ref{Lalpha}(ii).  
When $G$ is abelian, then $\alpha'=1$ and there are $\varphi(m)$ choices for $\gamma(c)$.  
This yields the count $\varphi(m)/\varphi(m/m')$ since $m'=m$ and since ${\rm Inn}(G)$ is trivial.  
If $G$ is non-abelian, then the image of $\gamma(c)$ in $M$ must have order $m$ 
and be congruent to $c$ modulo $\langle c^{m/m'} \rangle={\rm ker}(\nu)$.  
There are $p^n \varphi(m)/\varphi(m/m')$ choices for $\gamma(c)$.  
This yields the desired count, since there are $p^n$ inner automorphisms of $G$ which fix $\sigma$, 
namely conjugation by powers of $\sigma$. 
\end{proof}

\section{Main results} \label{Smain}

Let $G$ be a semi-direct product of the form $\ZZ/p^n \rtimes \ZZ/m$.
This section contains three results: first we prove that one can dominate a given Galois extension
having group $\ZZ/p^{n-1} \rtimes \ZZ/m$
by a Galois extension having group $\ZZ/p^n \rtimes \ZZ/m$, 
with control over the last upper jump; second, we give necessary and sufficient conditions 
for the ramification filtration of a $G$-Galois extension; third, we define a parameter space 
for $G$-Galois extensions of $K$ with given ramification filtration $\eta$ and calculate its dimension in terms of 
the upper jumps.

\subsection{A wild embedding problem}

We prove that one can embed a given Galois extension
having group $\ZZ/p^{n-1} \rtimes \ZZ/m$ by a Galois extension having group $\ZZ/p^n \rtimes \ZZ/m$,
with control over the last upper jump.  
See \cite[24.42]{FJ} for an earlier version of this result, in which $m=1$ and there is no 
control over the upper jump.  
Recall that $G/H_{n-1}$ is a semi-direct product of the form $\ZZ/p^{n-1} \rtimes \ZZ/m$.

\begin{proposition} \label{Pdominate}
Suppose $L_{n-1}/K$ is a $G/H_{n-1}$-Galois extension with upper jumps $u_1 \leq \cdots \leq u_{n-1}$.
Let $u_n \in \frac{1}{m}\NN$ be such that either $u_n = p u_{n-1}$ or both $u_n > pu_{n-1}$ and $p \nmid mu_n$.
Suppose also that $mu_n \equiv mu_1 \bmod m$.
Then there exists a $G$-Galois extension $L_n/K$ with upper jumps $u_1 \leq \cdots \leq u_n$ that dominates $L_{n-1}/K$.
\end{proposition}

\begin{proof}
Without loss of generality, one can suppose $L_0 \simeq K[x]/(x^m - 1/t)$ and $c(x) = \zeta x$.
The $\ZZ/p^{n-1}$-Galois extension $L_{n-1}/L_0$ has upper jumps $mu_1 \leq \cdots \leq mu_{n-1}$ 
by Lemma \ref{Lcomparejumps}.  
By Section \ref{SWitt}, $L_{n-1}/L_0$ is given by a Witt vector equation
$(y_1^p, \ldots, y_{n-1}^p) = (y_1, \ldots, y_{n-1}) +' (x_1, \ldots, x_{n-1})$
for some $x_i \in L_0$, such that $x_1$ is not a $\wp$th power in $L_0$.  
Furthermore,  
one can choose $(x_1, \ldots, x_{n-1})$ to be in standard form.
In particular, if $x_i \ne 0$, then $p \nmid v_0(x_i)$.

By Proposition \ref{Paction}, if $1 \leq i \leq n-1$, 
then $c(x_i) = \zeta^{j}x_i$ and $c(y_i) = \zeta^{j}y_i$ where $j=mu_1$.
By Lemma \ref{Lconductor}, $mu_{n-1}={\rm max}\{-p^{n-i}v_0(x_i)\}_{i=1}^{n-1}$.

If $u_n \ne pu_{n-1}$, let $x_n = x^{mu_n}$.  
In this case, $-v_0(x_n) = mu_n$. 
If $u_n = pu_{n-1}$, let $x_n = 0$.
In this case, $-v_0(x_n) = -\infty < pmu_{n-1}$.
In both cases, $(x_1, \ldots, x_n)$ is a Witt vector in standard form.
Then the Witt vector equation $(y_1^p, \ldots, y_n^p) = (y_1, \ldots, y_n) +' (x_1, \ldots, x_n)$ yields a 
$P$-Galois extension $L_n/L_0$ dominating $L_{n-1}/L_0$,
with upper jumps $mu_1 \leq \cdots \leq mu_n$ by Lemma \ref{Lconductor} (i.e., \cite[Thm.\ 1.1]{Ga:ls}).  

By the definition of $x_n$, then $c(x_n) = \zeta^{j}x_n$.  Let $c(y_n) = \zeta^{j}y_n$.
By Proposition \ref{Paction}, $L_n/K$ is a $G$-Galois extension dominating $L_{n-1}/K$,
and it has upper jumps $u_1 \leq \cdots \leq u_n$ by Lemma \ref{Lcomparejumps}.
\end{proof}

\subsection{Conditions on the ramification filtration} \label{Smaintheorem}

The ramification filtration of a Galois extension with group $G$ of the form $\ZZ/p^n \rtimes \ZZ/m$
is completely determined by either its lower or upper jumps.  
Here are the statement and proof of Theorem \ref{Tmaintheorem}, giving necessary and sufficient conditions
on the ramification filtrations of $G$-Galois extensions of $K$.

\begin{theorem} \label{Tnecsuff}
Let $G$ be a semi-direct product of the form $\ZZ/p^n \rtimes \ZZ/m$ where $p \nmid m$.
Let $\sigma \in G$ have order $p^n$ and let $m'=|{\rm Cent}_G(\sigma)|/p^n$.
A sequence $u_1 \leq \cdots \leq u_n$ of rational numbers
occurs as the set of positive breaks in the upper numbering of the ramification filtration
of a $G$-Galois extension of $k((t))$ if and only if:
\begin{description}
\item{(a)} $u_i \in \frac{1}{m} \NN$ for $1 \leq i \leq n$;
\item{(b)} $\gcd(m, mu_1)=m'$;
\item{(c)} $p \nmid mu_1$ and, for $1 < i \leq n$, either $u_i=pu_{i-1}$ or both $u_i > pu_{i-1}$ and $p \nmid mu_i$;
\item{(d)} and $mu_i \equiv mu_1 \bmod m$ for $1 \leq i \leq n$.
\end{description}
\end{theorem}

\begin{proof}
Conditions (a)-(d) are necessary:
let $u_1 \leq \cdots \leq u_n$ be the set of upper jumps of a $G$-Galois extension of $k((t))$.
The upper jumps of the $\ZZ/p^n$-subextension $L/L_0$ are $w_1 \leq \cdots \leq w_n$
where $w_i=mu_i$ by Lemma \ref{Lcomparejumps}.  
Condition (a) follows since $w_i \in \NN$ by the Hasse-Arf Theorem.
Condition (b) follows from Lemma \ref{Lgcd}(i).
Condition (c) is due to \cite{Schmid}, see Lemma \ref{Lcyclicjumps}.
Condition (d) follows from Proposition \ref{Pcongruence}(ii).

Conditions (a)-(d) are sufficient:
recall that $G$ has generators $\sigma$ (of order $p^n$) and $c$ (of order $m$) 
and $c\sigma c^{-1}=\sigma^{\alpha'}$ for some integer $\alpha'$ such that 
$1 \leq \alpha' < p^n$ and $p \nmid \alpha'$.
Let $\alpha \in \FF_p^* \simeq (\ZZ/p)^*$ be such that $\alpha \equiv \alpha' \bmod p$.
Let $j=mu_1$.
By condition (b), $\zeta^j$ has order $m/m'$ in $k^*$.
Likewise, $\alpha^{-1}$ has order $m/m'$ in $k^*$.
Thus there exists an integer $\beta$ such that $\zeta^{\beta j}=\alpha^{-1}$.

Consider the $\langle c \rangle$-Galois extension $L_0/K$ with equation $x^m=1/t$ and Galois action $c(x)=\zeta^\beta x$. 
Let $x_1 \in x^jk[[x^{-m}]]^*$.  
Consider the $\ZZ/p$-Galois extension $L_1/L$ with equation $y_1^p-y_1=x_1$ and Galois action $\sigma(y_1)=y_1+1$.
By \cite[Lemma 1.4.1]{Pr:deg}, $L_1/K$ is a $J_1$-Galois extension. 
It has lower jump $j$ and thus upper jump $u_1$.
By conditions (a), (c), (d), and Proposition \ref{Pdominate}, there exists a $G$-Galois extension $L/K$ 
dominating $L_1/K$ with upper jumps $u_1 \leq \cdots \leq u_n$.
\end{proof}

\begin{corollary} 
Let $G$ be a semi-direct product of the form $\ZZ/p^n \rtimes \ZZ/m$ where $p \nmid m$.
Suppose $\eta$ is a ramification filtration of $G$ satisfying conditions (a)-(d).
Let $f$ be the order of $p$ modulo $m/m'$ and let $q=p^f$.
Then there exists a $G$-Galois extension $L/K$ with ramification filtration $\eta$ which is defined over $\FF_q$. 
\end{corollary}

\begin{proof}
It suffices to produce a $G$-Galois extension $L/K$ whose equations and 
Galois action have coefficients in $\FF_q$.
Note that $\zeta^{j_1}$ has order $m/m'$ in $k^*$.
By the definition of $f$, the field $\FF_{p^f}$ contains the $(m/m')$th roots of unity, and thus contains $\zeta^{j_1}$.  
The case $n=1$ follows by direct computation with the equation $y_1^p-y_1=x_1^{mu_1}$, see \cite[Lemma 1.4.1]{Pr:deg}.
The result then proceeds by induction on $n$.  For the inductive step, one produces an equation for 
the extension $L/L_{n-1}$ using Proposition \ref{Pdominate}.  
In the proof of that result, recall that $x_n \in \FF_p[x]$ by definition.
Thus the equation has coefficients in $\FF_p$ by Lemma \ref{LWitt}.
The Galois action is defined over $\FF_q$ by (\ref{Eaction}) and Proposition \ref{Paction}.
\end{proof}

\subsection{Parameter space for $G$-Galois extensions} \label{Sparameter}

Given a sequence $u_1 \leq \cdots \leq u_n$ satisfying conditions (a)-(d),
let $\eta$ be the ramification filtration of $G$ having upper jumps $u_1 \leq \cdots \leq u_n$. 
By Theorem \ref{Tnecsuff}, there exists a $G$-Galois extension of $k((t))$ 
with ramification filtration $\eta$.
We prove there is a scheme ${\mathcal M}_\eta$ such that there is a natural bijection between the $k$-points of 
${\mathcal M}_\eta$ and isomorphism classes of $G$-Galois extensions of $k((t))$ with 
ramification filtration $\eta$.  We calculate the dimension of ${\mathcal M}_\eta$ in terms of
the sequence $u_1 \leq \cdots \leq u_n$.

\begin{notation}
Given positive integers $w$ and $m$, let 
$$\epsilon_p(w,m)=\#\{e \in \ZZ \ | \ 1 \leq e \leq w, \ e \equiv w \bmod m, \ p \nmid e\}.$$ 
\end{notation}

\begin{lemma}
Let $\delta_p(w,m) = 1$ if $w \equiv ap \bmod m$ for some $1 \leq a \leq r$, where 
$r$ is the remainder when $\lfloor w/p \rfloor$ is divided by $m$, and $\delta_p(w, m) = 0$ otherwise.
Then $\epsilon_p(w,m)=\lceil w/m \rceil - \lfloor w/mp \rfloor - \delta_p(w,m)$. 
\end{lemma}

\begin{proof}
The number of integers $e$ such that $1 \leq e \leq w$ and $e \equiv w \bmod m$ is $\lceil w/m \rceil$.  
To count the number of these which are divisible by $p$, 
consider the set $A = \{p, 2p, \ldots, \lfloor w/p \rfloor p \}$.
Then $A$ contains at least $\lfloor\lfloor w/p \rfloor /m \rfloor = \lfloor w/mp \rfloor$ elements $e$ such that  $e \equiv w \bmod m$. 
Let $r$ be the remainder when $\lfloor w/p \rfloor$ is divided by $m$.  
Then $A$ contains one additional element $e \equiv w \bmod m$ if and only if an element of 
$\{p, 2p, \ldots, rp\}$ is congruent to $w$ modulo $m$.
The formula holds since $\delta_p(w,m) = 1$ precisely in this case.
\end{proof}

Given a positive integer $N$, the root of unity $\zeta_{m/m'}$ acts on the affine variety ${\mathbb A}^N$ 
via multiplication on each 
coordinate.  Let ${\mathbb A}^N/\mu_{m/m'}$ denote the quotient.

\begin{theorem} \label{Tparameter}
Let $G$ be a semi-direct product of the form $\ZZ/p^n \rtimes \ZZ/m$ where $p \nmid m$.
Let $u_1 \leq \cdots \leq u_n$ be a sequence satisfying conditions (a)-(d)
and $\eta$ be the ramification filtration of $G$ with upper jumps $u_1 \leq \cdots \leq u_n$. 
Let $N_\eta=\sum_{i=1}^n \epsilon_p(mu_i, m)$.  
Then there is an open subscheme $U_\eta \subset {\mathbb A}^{N_\eta}/\mu_{m/m'}$ 
and a finite \'{e}tale map $\pi: {\mathcal M}_{\eta} \to U_{\eta}$ of degree $\varphi(m)/\varphi(m/m')$
such that the $k$-points of ${\mathcal M}_{\eta}$ are in natural bijection with isomorphism classes of 
$G$-Galois extensions of $k((t))$ with ramification filtration $\eta$.
\end{theorem}

It is clear that ${\rm dim}({\mathcal M}_\eta)=N_\eta$ depends only on $p, m, u_1, \ldots, u_n$.

\begin{proof}
By Lemma \ref{Lisom}, it suffices to show that the collection of Witt vectors $(x_1, \ldots, x_n)$ in standard
form, which, as in Proposition \ref{Paction}, yield $G$-Galois extensions $L/K$ with ramification invariants 
$u_1 \leq \cdots \leq u_n$, is in natural bijection with the $k$-points of an open subscheme of ${\mathbb A}^{N_{\eta}}$.

The proof is by induction on $n$.  
For the case $n=1$, Lemma \ref{Lconductor} shows that $x_1 \in k[x]$ must have degree $mu_1$. 
By Proposition \ref{Paction}, the extension $L_1/K$ is $J_1$-Galois if and only if 
$c(x_1) = \zeta^{mu_1}x_1$, in other words, if and only if all exponents of $x_1$ are congruent 
to $mu_1$ modulo $m$.  Since $x_1$ is in standard form, it has no exponents with degree divisible by $p$.  
Thus the number of possible exponents is $\epsilon = \epsilon_p(mu_1, m)$.  
Since the leading coefficient of $x_1$ is nonzero, the choice of $x_n$ is equivalent to the choice of a 
$k$-point in an open subscheme of ${\mathbb A}^{\epsilon}$.  (See also \cite[Proposition 2.2.6]{Pr:deg}).

Now, suppose that $(x_1, \ldots, x_{n-1})$ is a Witt vector in standard form, which yields a
$G/H_{n-1}$-Galois extension $L_{n-1}/K$ with upper jumps $u_1 \leq \cdots \leq u_{n-1}$.   
Let $\epsilon=\epsilon_p(mu_n, m)$.
It suffices to show that Witt vectors $(x_1, \ldots, x_n)$ in standard form which yield 
an extension $L/K$ dominating 
$L_{n-1}/K$ with upper jumps $u_1 \leq \cdots \leq u_n$ are in natural bijection with the $k$-points of an open subscheme 
$\tilde{U}_n \subset {\mathbb A}^\epsilon$.

The Witt vector $(x_1, \ldots, x_n)$ for the extension $L/K$ is determined by the choice of
$x_n \in k[x]$ in standard form. 
By Proposition \ref{Paction}, the extension $L/K$ is $G$-Galois if and only if 
$c(x_n) = \zeta^{mu_1}x_n$, in other words, if and only if all exponents of $x_n$ are congruent 
to $mu_1$ modulo $m$.  Recall that $mu_1 \equiv mu_n \bmod m$ by Proposition \ref{Pcongruence}.

By Lemma \ref{Lconductor}, the extension $L/K$ has upper jump $u_n$ if and only if
${\rm deg}(x_n)=-v_0(x_n) \leq mu_n$, where equality must hold if $u_n > pu_{n-1}$.
Thus, an exponent $e$ appearing in $x_n$ satisfies $0 \leq e \leq mu_n$, and $e \equiv mu_n \bmod m$, and $p \nmid e$.  
The number of these exponents is 
$\epsilon = \epsilon_p(mu_n, m)$.  The leading coefficient of $x_n$ must be non-zero when $u_n > pu_{n-1}$.
The choice of $x_n$ is thus equivalent to the choice of a $k$-point in an open subscheme of ${\mathbb A}^{\epsilon}$.
\end{proof}

\begin{remark}
Consider the contravariant functor $F_\eta$ from the category of schemes to sets, which associates to a scheme
$B$ the set of $G$-Galois extensions of ${\mathcal O}_B((t))$ whose geometric fibres have ramification filtration $\eta$.
The scheme ${\mathcal M}_\eta$ does not represent $F_\eta$ on the category of $k$-schemes because
there are non-constant $G$-Galois covers defined over a base scheme $B$, which become constant
after pullback by a finite morphism $B' \to B$. 
The scheme ${\mathcal M}_\eta$ is a fine moduli space for $F_\eta$ 
on a category where such morphisms are trivialized; see \cite[Thm.\ 2.2.10]{Pr:deg} for the case $n=1$.
\end{remark}

\begin{remark}
In \cite[Prop.\ 4.1.1]{BM}, the authors calculate the dimension of the tangent space of the versal 
deformation space of a $\ZZ/p^n$-Galois extension in terms of its ramification filtration.  
Theorem \ref{Tparameter} is less technical than their result and it is not clear how to compare them
directly. 
\end{remark} 

\section{Equations for $\ZZ/p^3$-Galois extensions} \label{Sp3}

It is well-known that the methods of Section \ref{SWitt} can be used to find equations for $\ZZ/p^n$-extensions \cite{Schmid2}, 
but the equations themselves are difficult to find in the literature.
Here are formulae for the general $\ZZ/p^3$-Galois extension of $K$.

\begin{example} \label{Pp3}
Suppose $L/K$ is a $\ZZ/p^3$-Galois extension of $K \cong k((t))$.
Then there exist $x_1, x_2, x_3 \in K$ so that $L/K$ is isomorphic to the
following extension:
\begin{eqnarray*}
y_1^p - y_1 &=& x_1;\\
y_2^p - y_2 &=& \frac{x_1^p + y_1^p - (x_1+y_1)^p}{p} + x_2; \\
y_3^p - y_3 &=& \frac{x_1^{p^2} + y_1^{p^2} - (x_1+y_1)^{p^2}}{p^2} +
\frac{x_2^p + y_2^p - (x_2 + y_2 + \frac{x_1^p + y_1^p - (x_1+y_1)^p}{p})^p}{p} + x_3.
\end{eqnarray*}
A generator $\sigma$ of the Galois group can be chosen so that its action is given by:
\begin{eqnarray*}
\sigma(y_1) &=& y_1 + 1; \\
\sigma(y_2) &=& y_2 + \frac{y_1^p + 1 - (y_1+1)^p}{p}; \\
\sigma(y_3) &=& y_3 + \frac{y_1^{p^2} + 1 - (y_1+1)^{p^2}}{p^2} +
\frac{y_2^p - (y_2 + \frac{y_1^p + 1 - (y_1+1)^p}{p})^p}{p}.
\end{eqnarray*}
\end{example}

The integral coefficients in Example \ref{Pp3} can be considered to be in $\FF_p \subset k$.

\begin{proof}
For the equations, it suffices to recursively compute $f_i=\overline{S}_{i-1}-y_i$ for $1 \leq i \leq 3$, 
starting with $S_0(x_1, y_1) = x_1 + y_1$ and $S_1(x_1, x_2, y_1, y_2) = x_2 + y_2 + (x_1^p +y_1^p - (x_1 + y_1)^p)/p$.
The Galois action is given by
$\sigma(y_i) = y_i + \tilde{f_i}$, where $\tilde{f_i} = f_i(y_1, \ldots, y_{i-1}, 1, 0, \ldots, 0)$.
To see this, note that $y_i^p = y_i + f_i$ and (\ref{EWitt}) imply that
$(y_1 + f_1, \ldots, y_n + f_n) = (y_1, \ldots, y_n) +' (x_1, \ldots, x_n)$.
Substituting $(1,0,\ldots, 0)$ for $(x_1, \ldots, x_n)$ yields
$(y_1 + \tilde{f_1}, \ldots, y_n + \tilde{f_n}) = (y_1, \ldots, y_n) +' (1, 0, \ldots, 0)$,
which equals $\sigma(y_1, \ldots, y_n)$ by Lemma \ref{Lgalois}.
\end{proof}

\begin{example}
When $p=2$ and $x=t^{-j}$, here are equations for a $\ZZ/8$-Galois extension of $k((t))$, 
which is defined over $\FF_2$ and has upper jumps $j$, $2j$, and $4j$:
$$y^2 - y = x; \ z^2 - z = xy; \ w^2 - w = x^3y+y^3x + xyz.$$
The Galois action is given by $y \mapsto y+1$, $z \mapsto z+y$, and $w \mapsto w + y^3 + y + yz$.
\end{example}

\bibliographystyle{abbrv}
\bibliography{pries_obus}

\end{document}